\documentclass[12pt]{article}

\usepackage{fullpage}
\usepackage{amssymb}
\usepackage{amsmath}
\usepackage{amsthm}
\usepackage[all]{xy}
\usepackage{graphicx}
\usepackage{mathrsfs}
\usepackage[colorlinks=true,urlcolor=blue]{hyperref}
\usepackage{verbatim}
\usepackage[capitalise,noabbrev]{cleveref}

\newtheorem{theorem}{Theorem}[section]

\newtheorem{proposition}[theorem]{Proposition}
\newtheorem{corollary}[theorem]{Corollary}
\newtheorem{definition}[theorem]{Definition}
\newtheorem{example}[theorem]{Example}

\newtheorem{problem}{Problem}

\newcommand{\G}{\Gamma}
\newcommand{\calG}{\mathcal G}
\newcommand{\calF}{\mathcal F}
\newcommand{\calM}{\mathcal M}
\newcommand{\calP}{\mathcal P}
\newcommand{\calQ}{\mathcal Q}

\newcommand{\mix}{\diamond}

\newcommand{\shm}{\!-\!}

\newcounter{reminder}

\begin{document}

\title{Flat extensions of abstract polytopes}

\author{Gabe Cunningham\\
Department of Mathematics\\
University of Massachusetts Boston\\
Boston, Massachusetts, USA, 02125 \\
gabriel.cunningham@gmail.com
}

\date{ \today }
\maketitle

\begin{abstract}

We consider the problem of constructing an abstract $(n \!+\! 1)$-polytope $\calQ$ with $k$ facets isomorphic
to a given $n$-polytope $\calP$, where $k \geq 3$. In particular, we consider the case where we want
$\calQ$ to be $(n \!-\! 2,n)$-flat, meaning that every $(n \!-\! 2)$-face is incident to every $n$-face (facet).
We show that if $\calP$ admits such a \emph{flat extension} for a given $k$, then the facet graph of
$\calP$ is $(k \!-\! 1)$-colorable. Conversely, we show that if the facet graph is $(k \!-\! 1)$-colorable
and $k \!-\! 1$ is prime, then $\calP$ admits a flat extension for that $k$. We also show that if $\calP$ is facet-bipartite,
then for every even $k$, there is a flat extension $\calP|k$ such that every automorphism of $\calP$
extends to an automorphism of $\calP|k$. Finally, if $\calP$ is a facet-bipartite $n$-polytope and $\calQ$ is 
a vertex-bipartite $m$-polytope, we describe a \emph{flat amalgamation} of $\calP$ and $\calQ$, an
$(m \!+\! n \!-\! 1)$-polytope that is $(n \!-\! 2,n)$-flat, with $n$-faces isomorphic to $\calP$ and co-$(n \!-\! 2)$-faces
isomorphic to $\calQ$.




\end{abstract}

\section{Introduction}

	Fix an abstract $n$-polytope $\calP$ and a positive integer $k$, and suppose that you want to glue together copies
	of $\calP$ to build an $(n \!+\! 1)$-polytope $\calQ$ such that each $(n \!-\! 2)$-face of $\calQ$ is 
	surrounded by $k$ copies of $\calP$. What is the smallest possible $\calQ$? 
	
	Clearly, the best we could hope for is to use only $k$ copies of $\calP$, building $\calQ$ so that
	every $(n \!-\! 2)$-face is surrounded by every copy of $\calP$. For which polytopes $\calP$ and which integers $k$ is this possible?
	When $k = 2$, this is always possible; this is called the \emph{trivial extension} of $\calP$. More generally,
	we will show that if $k$ is even, then this is always possible if $\calP$ is \emph{facet-bipartite} (in other words,
	if we can color the facets with two colors such that adjacent facets have different colors). On the other
	hand, we will show that if $\calP$ is not $(k \!-\! 1)$-facet-colorable, then it is impossible to glue together
	$k$ copies of $\calP$ in this manner.

	The polytopes that we are working with are \emph{abstract polytopes}, which are usually defined in terms of a poset
	that is similar to the face-lattice of a polytope \cite{arp}. For the constructions discussed here, it is more natural to
	consider polytopes as a subclass of \emph{maniplexes}, which can be viewed as a kind of edge-colored graphs
	\cite{maniplexes}. The paper \cite{poly-mani} provides a characterization of which maniplexes are the flag graphs of polytopes, 
	which is a key ingredient to our approach.

	We start by giving some background on maniplexes and polytopes in Section~2. Then we consider the problem of building
	a flat extension of $\calP$ that uses $k$ copies in Section~3. \cref{facet-coloring} shows that the facet graph of $\calP$ must
	be $(k \!-\! 1)$-colorable. 
	In \cref{flat-exts1}, we will show that if $\calP$ is facet-bipartite, then any even $k \geq 2$ will work (see
	\cref{pk-polytope}) and we determine some further properties related to its automorphism group (see
	\cref{pk-auts}). Then, in \cref{other-exts}, we describe a more general construction that works for
	any $\calP$ whose facet graph is $(k \!-\! 1)$-colorable, subject to some restrictions
	on $k$ (see \cref{pk-polytope-2}). In \cref{flat-amalg}, we generalize the first construction in another way,
	building a flat amalgamation of a facet-bipartite polytope $\calP$ and a vertex-bipartite polytope $\calQ$.
	Finally, we briefly discuss some open questions that remain in \cref{open-q}.
	
\section{Maniplexes and polytopes}

	Abstract polytopes are posets that, broadly speaking, look something like the incidence relation of
	a convex polytope or a tiling of a surface or space. Their basic theory is outlined in \cite{arp}.
	Another way to view a polytope is in terms of its flag graph, and in \cite{poly-mani}, Garza-Vargas 
	and Hubard characterize which properly-edge-colored regular simple graphs are the flag
	graphs of abstract polytopes. Since the constructions in this paper operate on the flag graphs
	of polytopes, it will be natural for us to define polytopes in terms of graphs instead of posets.
	
	Let us start with a (non-standard) definition. Let $\calG$ be a graph whose nodes we
	will call \emph{flags}. Then $\calG$ is an \emph{$n$-pre-maniplex}
	if it is an $n$-regular simple graph where the edges are colored $\{0, 1, \ldots, n \!-\! 1\}$
	and each flag is incident to exactly one edge of each color. For each color $i$ and each
	flag $\Phi$, we define $\Phi^i$ to be the other endpoint of the edge of color $i$ that touches $\Phi$,
	and we say that $\Phi^i$ is \emph{$i$-adjacent} to $\Phi$. We further define $\Phi^{i,j}$ to be
	$(\Phi^i)^j$. 
	
	If $\calG$ is an $n$-pre-maniplex, then let $\calG[i_1, \ldots, i_m]$ denote the subgraph of $\calG$
	with all of the same flags as $\calG$ and with only the edges of colors $i_1, \ldots, i_m$. 
	The \emph{$(i_1, \ldots, i_m)$-color-components} of $\calG$ are the connected components of
	$\calG[i_1, \ldots, i_m]$.
	
	In an $n$-pre-maniplex $\calG$, we say that colors $i$ and $j$ \emph{commute} if, for each flag $\Phi$, $\Phi^{i,j} = \Phi^{j,i}$.
	Equivalently, $i$ and $j$ commute if $\calG[i,j]$ is a union of $4$-cycles. Note that if $A$ and $B$ are sets of colors such that
	every color in $A$ commutes with every color in $B$, then whenever there is a path from $\Phi$ to $\Psi$ using edges of colors
	in $A \cup B$, there must be a flag $\Lambda$ such that there is a path from $\Phi$ to $\Lambda$ using color set $A$ and then
	a path from $\Lambda$ to $\Psi$ using color set $B$.
	
	We define an \emph{$n$-maniplex} to be an $n$-pre-maniplex such that, for every pair of colors
	$i$ and $j$ such that $|i-j| > 1$, those colors commute. For each $i \in \{0, \ldots, n \!-\! 1\}$, 
	the \emph{$i$-faces} of an $n$-maniplex are the connected components of $\calG[0, \ldots, i-1, i+1, \ldots, n \!-\! 1]$.
	We say that two faces are \emph{incident} if they have nonempty intersection. The $(n \!-\! 1)$-faces of an $n$-maniplex
	are called its \emph{facets}. 
	
	Finally, an $n$-maniplex is an \emph{$n$-polytope} if it satisfies the following \emph{Path Intersection Property}:
	for every pair of flags $\Phi$ and $\Psi$ and every $i < j$, if there is a path between $\Phi$ and $\Psi$ 
	that uses colors $i, \ldots, n \!-\! 1$ and another path between them that uses colors $0, \ldots, j$,
	then there must be a path between them that uses only the colors $i, \ldots, j$ (see \cite[Thm. 5.3]{poly-mani}).
	
	In the context of graphs, an \emph{automorphism} of an $n$-polytope is a graph automorphism that
	preserves the edge colors, and we denote the automorphism group of $\calP$ by $\G(\calP)$.
	In other words, $\varphi$ is an automorphism of $\calP$ if it is a bijection on the flags such that,
	for every flag $\Phi$ and every edge color $i$, we have $\Phi^i \varphi = (\Phi \varphi)^i$.
	If $\calP$ and $\calQ$ are $n$-polytopes, then $\calP$
	\emph{covers} $\calQ$ if there is a surjective graph homomorphism from $\calP$ to $\calQ$ that
	preserves the edge colors. A polytope is \emph{regular} if the automorphism group acts transitively
	on the flags. The \emph{symmetry type graph} of a polytope $\calP$ is the quotient
	of $\calP$ by the orbits of the nodes under $\G(\calP)$; see \cite{stg}.
	
	The \emph{facet graph} of a polytope $\calP$ is a simple graph whose nodes correspond to the facets of
	$\calP$, and where two nodes are connected if the corresponding facets are connected by an edge
	labeled $n \!-\! 1$ in $\calP$. 
	A polytope is \emph{facet-bipartite} if its facet graph is bipartite. Equivalently, a polytope is facet-bipartite
	if and only if there are no cycles in $\calP$ with an odd number of edges labeled $n \!-\! 1$.

	The \emph{dual} of a polytope $\calP$ is the polytope $\calP^{*}$ obtained by changing every edge
	label from $i$ to $n \!-\! 1-i$. The \emph{$1$-skeleton} of $\calP$ is the facet graph of $\calP^{*}$. That is,
	the nodes of the $1$-skeleton correspond to the $0$-faces of $\calP$, and two nodes are connected
	if there is an edge labeled $1$ between the corresponding faces in $\calP$. The polytope $\calP$ is
	\emph{vertex-bipartite} if there are no cycles in $\calP$ with an odd number of edges labeled $0$.

	A polytope $\calP$ is \emph{$(i,j)$-flat} if every $i$-face is incident to every $j$-face. In other words, 
	$\calP$ is $(i,j)$-flat if, for every flag $\Phi$ and every $j$-face, there is a path from $\Phi$ to some flag in
	that $j$-face that does not use any edges of color $i$.
	
	\begin{proposition}
	\label{flat-crit}
	Suppose $i < j$. Then the $n$-polytope $\calP$ is $(i,j)$-flat if and only if, for every flag $\Phi$ and every $j$-face,
	there is a path from $\Phi$ to some flag in that $j$-face that only uses edges of colors $\{i+1, \ldots, n \!-\! 1\}$.
	\end{proposition}
	
	\begin{proof}
	Suppose that $\calP$ is $(i,j)$-flat and consider an arbitrary flag $\Phi$ and a $j$-face. Suppose that
	$\Psi$ is a flag in the $j$-face such that there is a path from $\Phi$ to $\Psi$ that never uses color $i$. 
	So the path from $\Phi$ to $\Psi$ uses colors $\{0, \ldots, i-1\}$ and $\{i+1, \ldots, n \!-\! 1\}$. Since these
	two color sets commute, there must be a flag $\Lambda$ such that there is a path from $\Phi$ to $\Lambda$
	using colors $\{i+1, \ldots, n \!-\! 1\}$ and then a path from $\Lambda$ to $\Psi$ using colors $\{0, \ldots, i-1\}$.
	Since $i < j$, the latter color set does not include $j$, and so $\Lambda$ is in the same $j$-face as $\Psi$.
	Then there is a path from $\Phi$ to the $j$-face that only uses edges of colors $\{i+1, \ldots, n \!-\! 1\}$.
	That proves one direction, and the other direction is clear.
	\end{proof}
	
\section{Flat extensions}

	Our goal is to take $k$ copies of an $n$-polytope $\calP$ and glue them together into an $(n \!+\! 1)$-polytope
	$\calQ$. Furthermore, we would like for every $(n \!-\! 2)$-face of $\calQ$ to be surrounded by all $k$ copies
	of $\calP$ --- in other words, we would like $\calQ$ to be $(n \!-\! 2, n)$-flat. How do we get started?

	If such a polytope $\calQ$ exists, then removing all edges labeled $n$ yields $k$ copies of $\calP$.
	So in order to build $\calQ$, let us take $k$ copies of $\calP$ (which we will call the \emph{layers}
	of $\calQ$), labeled $\calP_1, \ldots, \calP_k$.
	For each flag $\Phi$ of $\calP$, we will write $\Phi_i$ for the image of $\Phi$ in $\calP_i$.
	Now, we create $\calQ$ from these $k$ copies of $\calP$ by adding a perfect matching using new
	edges labeled $n$. How do we do so in a way that ensures that $\calQ$ is a polytope?
	
	First we need to make sure that color $n$ commutes with each color $c$ in $\{0, \ldots, n \!-\! 2\}$.
	To do so, once we decide to match some flag
	$\Phi_i$ to $\Psi_j$, we must also match $(\Phi_i)^c$ to $(\Psi_j)^c$ for every $c \in \{0, \ldots, n \!-\! 2\}$. 
	Applying this restriction recursively shows that the matching of flags must induce a matching of 
	the $\{0,\ldots,n \!-\! 2\}$-color components, which correspond to the facets of $\calP$. (See \cref{induced-matching}.)
	
	\begin{figure}[htbp]
	\begin{center}
	\includegraphics[height=6cm]{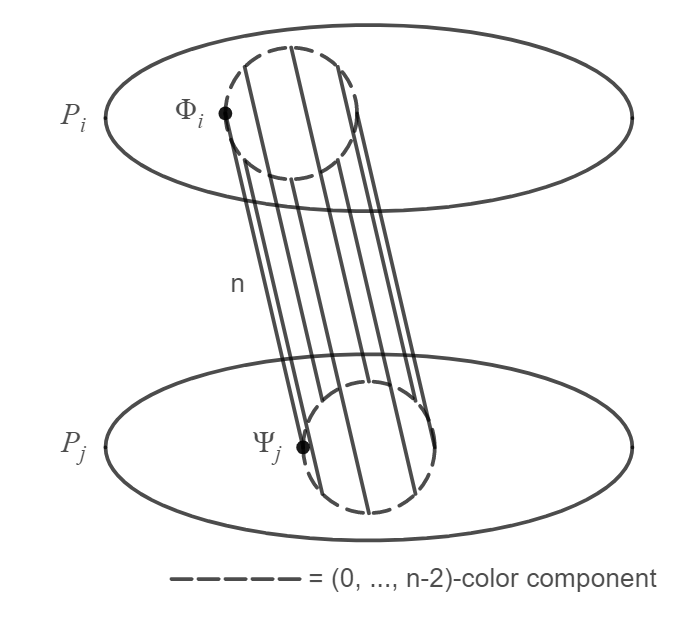}
	\caption{Matching $\Phi_i$ to $\Psi_j$ induces a matching of the $(0, \ldots, n \!-\! 2)$-color components.}
	\label{induced-matching}
	\end{center}
	\end{figure}

	Next, we want $\calQ$ to be $(n \!-\! 2, n)$-flat. By \cref{flat-crit}, this is equivalent to making every
	$(n \!-\! 1,n)$-color-component intersect every $\calP_j$. 

	We have already observed that once we match a flag $\Phi$, that induces a matching of $\Phi^c$ for each
	$c \in \{ 0, \ldots, n \!-\! 2 \}$. Now we will see that requiring that $\calQ$ be flat restricts
	our choice of how we match $\Phi^{n \!-\! 1}$.

	\begin{proposition} \label{different-layers}
	Suppose $\calQ$ is an $(n \!+\! 1)$-polytope that is $(n \!-\! 2,n)$-flat, with $k$ facets isomorphic to $\calP$, where
	$k \geq 3$. Then for every $\Phi_i$, the flags $(\Phi_i)^{n}$ and $(\Phi_i)^{n \!-\! 1,n}$ are in different layers $\calP_j$.
	\end{proposition}
	
	\begin{proof}
	Suppose $(\Phi_i)^n$ and $(\Phi_i)^{n \!-\! 1,n}$ are in the same layer. Then there is a path from $(\Phi_i)^n$ to $(\Phi_i)^{n \!-\! 1, n}$ 
	using edges labeled $\{0, \ldots, n \!-\! 1\}$. There is also a path from $(\Phi_i)^n$ to $(\Phi_i)^{n \!-\! 1,n}$ using edges labeled
	only $n \!-\! 1$ and $n$. Then the Path Intersection Property implies that there is a path using only edges labeled
	$n \!-\! 1$, which means that $(\Phi_i)^{n,n \!-\! 1} = (\Phi_i)^{n \!-\! 1,n}$. Thus the $(n \!-\! 1,n)$-color component that contains
	$\Phi_i$ consists of only four flags in two layers, and since $k \geq 3$ this implies that $\calQ$ is not $(n \!-\! 2,n)$-flat.
	\end{proof}

	Let us reinterpret this result in terms of the facet graph of $\calP$. For each facet of $\calP$ (corresponding to
	a $(0, \ldots, n \!-\! 2)$-color component of $\calQ$), consider the
	flags in the last layer $\calP_k$ that contain that facet. By the discussion earlier, all of these flags are matched to flags
	in some single layer $\calP_i$ with $i \in \{1, \ldots, k \!-\! 1\}$. Then we may color each facet of $\calP$
	by that number $i$, and \cref{different-layers} implies that this is a \emph{proper} coloring!
	Therefore,
	
	\begin{corollary} \label{facet-coloring}
	Let $k \geq 3$. If $\calP$ is an $n$-polytope such that its facet graph is not $(k \!-\! 1)$-colorable, then there are no $(n \!+\! 1)$-polytopes $\calQ$ with $k$ facets
	isomorphic to $\calP$ such that $\calQ$ is $(n \!-\! 2,n)$-flat.
	\end{corollary}

	\begin{example}
	Since the facet graph of the $n$-simplex is the complete graph $K_{n \!+\! 1}$, there are no $(n \!-\! 2,n)$-flat $(n \! + \! 1)$-polytopes $\calQ$ with
	$n \!+\! 1$ simplicial facets.
	\end{example}
	
	\subsection{Flat extensions of facet-bipartite polytopes}
	\label{flat-exts1}

		When trying to define a matching in order to build $\calQ$, the most straightforward way would be for each
		$\Phi_i$ to be matched to some $\Phi_j$. That is, each flag is matched to the `same' flag in a different
		layer. The easiest such matching would have each flag $\Phi_i$ matched to either $\Phi_{i-1}$ or $\Phi_{i+1}$.
		Then the argument for \cref{facet-coloring} works in essentially the same way to show that, since
		each layer is matched to only two other layers, $\calP$ must be facet-bipartite in order for this to work.
		We will show that this necessary condition is also sufficient.
		
		So, suppose that $\calP$ is a facet-bipartite $n$-polytope, and let $k$ be an even positive integer. 
		Given a proper coloring of the facet graph of $\calP$ with two colors (say red and blue), we can
		color each flag of $\calP$ according to the color of its facet. Then, for each red flag $\Phi$, we
		will match $\Phi_1$ to $\Phi_2$, $\Phi_3$ to $\Phi_4$, and so on. For each blue flag $\Psi$,
		we will match $\Psi_2$ to $\Psi_3$, $\Psi_4$ to $\Psi_5$, and so on (matching $\Psi_k$ to $\Psi_1$).
		We refer to the graph that we obtain by $\calP|k$. (See \cref{flat-ext-figure}.)
		
		\begin{figure}[htbp]
		\begin{center}
		\includegraphics[height=5cm]{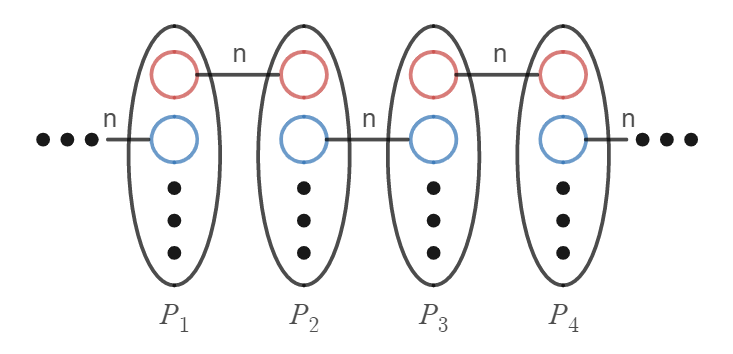}
		\caption{Flags are matched according to the coloring of the facet graph of $\calP$.}
		\label{flat-ext-figure}
		\end{center}
		\end{figure}

		First, let us show that this construction really yields a polytope with the desired properties.
		
		\begin{theorem}
		\label{pk-polytope}
		The graph $\calP|k$ is the flag graph of an $(n \!-\! 2,n)$-flat $(n \!+\! 1)$-polytope with $k$ facets
		isomorphic to $\calP$.
		\end{theorem}
		
		\begin{proof}
		By construction, it is clear that $\calP|k$ has $k$ facets isomorphic to $\calP$. If $\Phi$ is a red flag
		and $\Psi = \Phi^{n \!-\! 1}$, then $\Psi$ is blue and the $(n \!-\! 1,n)$-color component that contains $\Phi_1$ is the cycle
		\[ (\Phi_1, \Phi_2, \Psi_2, \Psi_3, \Phi_3, \Phi_4, \ldots, \Psi_k, \Psi_1), \]
		which intersects every layer. It is clear then that each 
		$(n \!-\! 1,n)$-color component intersects every $\calP_i$, and so $\calP|k$ is $(n \!-\! 2,n)$-flat.
		It is also clear that $\calP|k$ is a maniplex, since we forced the new edges labeled $n$ to
		commute with the edges labeled $0, 1, \ldots, n \!-\! 2$. 
		
		It remains to show that $\calP|k$ is a polytope by showing that it satisfies the Path Intersection Property.
		Consider colors $i$ and $j$ satisfying $0 \leq i < j \leq n$. Suppose there are two flags such that there
		is a path between them using colors $0, \ldots, j-1$ and $i+1, \ldots, n$. Since $j-1 < n$, it follows
		that the two flags are in the same layer, and without loss of generality we will assume they are in layer 1.
		So there are two flags $\Phi_1$ and $\Psi_1$ such that there is a path between them that uses
		colors $i+1, \ldots, n$. Since edges of color $n$ always connect two flags with the same underlying
		flag in $\calP$, such a path induces a path between $\Phi$ and $\Psi$ in $\calP$ that only
		uses colors $i+1, \ldots, n \!-\! 1$. Similarly, there is an induced path between $\Phi$ and $\Psi$ in
		$\calP$ that uses colors $0, \ldots, j-1$. Then, since $\calP$ is a polytope, it follows that
		there is a path from $\Phi$ to $\Psi$ that uses colors $i+1, \ldots, j-1$, and then this
		path also lifts to an isomorphic path from $\Phi_1$ to $\Psi_1$ using only those colors,
		as desired.
		\end{proof}

		\begin{example}
		If $\calP$ is the unique $1$-polytope, then $\calP|k$ is a $k$-gon.
		\end{example}

		\begin{example}
		If $\calP$ is a square, then $\calP|4$ is the map $\{4, 4\}_{(2,0)}$ on the torus (see \cite[Sec. 1D]{arp}).
		\end{example}
		
		\begin{example}
		If $k = 2$, then we don't even need for $\calP$ to be facet-bipartite --- we can just match each
		$\Phi_1$ to $\Phi_2$. Indeed, $\calP|2$ is the \emph{trivial extension} of $\calP$, also denoted
		$\{\calP, 2\}$.
		\end{example}
		
		\begin{example}
		Nothing goes wrong if we try $k = \infty$ and index the layers $\calP_i$ by letting $i$ be any integer.
		We still get an $(n \!-\! 2,n)$-flat polytope with infinitely many facets isomorphic to $\calP$.
		\end{example}

		Now let us determine the automorphism group of $\calP|k$. Fix a base flag $\Phi$ of $\calP$,
		and consider an automorphism $\varphi$ of $\calP$ that sends $\Phi$ to $\Psi$. Can we extend
		$\varphi$ to an automorphism $\tilde{\varphi}$ of $\calP|k$?
		
		Without loss of generality, let us assume that $\Phi$ is red. Then the other red flags are those that
		can be reached from $\Phi$ using an even number of edges labeled $n \!-\! 1$, and the blue flags
		are those that can be reached from $\Phi$ using an odd number of edges labeled $n \!-\! 1$.
		Furthermore, $\varphi$ respects these color classes since, for each flag $\Lambda$,
		we have $\Lambda^{n \!-\! 1} \varphi = (\Lambda \varphi)^{n \!-\! 1}$. 

		Now, if $\Psi$ is also red, then $\varphi$ preserves the color of every flag. 
		Then we define $\tilde{\varphi}$ so that, for each flag $\Lambda$ of $\calP$,
		\[ (\Lambda_i) \tilde{\varphi} = (\Lambda \varphi)_i. \]
		In other words, $\tilde{\varphi}$ fixes each layer setwise, and acts on each layer in the
		same way that $\varphi$ acts on $\calP$. To see that this defines an automorphism, it suffices
		to show that $\tilde{\varphi}$ preserves the edges of color $n$, and this is true since
		\[ (\Lambda_i)^n \tilde{\varphi} = \Lambda_{i \pm 1} \tilde{\varphi} = (\Lambda \varphi)_{i \pm 1} = ((\Lambda \varphi)_i)^n = (\Lambda_i \tilde{\varphi})^n. \]
	
		If $\Psi$ is blue instead, then the action of
		$\varphi$ on $\calP$ reverses the color of every flag. Then we define $\tilde{\varphi}$ 
		so that, for each flag $\Lambda$ of $\calP$,
		\[ (\Lambda_i) \tilde{\varphi} = (\Lambda \varphi)_{k+2-i}. \]
		Again, this will define an automorphism if and only if $\tilde{\varphi}$ preserves the edges of color $n$, and this is true since
		\[ (\Lambda_i)^n \tilde{\varphi} = \Lambda_{i \pm 1} \tilde{\varphi} = (\Lambda \varphi)_{k+2-i \mp 1} = ((\Lambda \varphi)_{k+2-i})^n = (\Lambda_i \tilde{\varphi})^n, \]
		where the third equality follows because $\Lambda \varphi$ is the opposite color of $\Lambda$, and so the matching
		of $\Lambda \varphi$ is in the opposite direction of the matching of $\Lambda$ (that is, $\mp$ instead of $\pm$).
		So in either case, we see that each automorphism of $\calP$ lifts to an automorphism of $\calP|k$; 
		in other words, $\calP|k$ is \emph{hereditary} (see \cite{hereditary}).

		In addition to these automorphisms $\tilde{\varphi}$, which all fix the first layer setwise, there are automorphisms
		of $\calP|k$ that simply permute the layers. Indeed, it is clear from the symmetry of the graph (see 
		\cref{flat-ext-figure}) that there is an automorphism $\alpha$ that sends each $\Lambda_i$ to $\Lambda_{k+3-i}$
		and an automorphism $\beta$ that sends each $\Lambda_i$ to $\Lambda_{k+5-i}$ (with the subscripts of $\Lambda$
		reduced modulo $k$). The subgroup $\langle \alpha, \beta \rangle$ acts transitively on the layers, and the orbit of
		the flag $\Lambda_1$ is all flags $\Lambda_j$.
		
		We can now characterize the automorphism group of $\calP|k$.
		
		\begin{proposition} \label{pk-auts}
		Let $\calP$ be a facet-bipartite $n$-polytope and let $k$ be a positive even integer.
		Let $\tilde{\varphi}$, $\alpha$ and $\beta$ be defined as above. 
		\begin{enumerate}
		\item $\calP|k$ is hereditary.
		\item $\G(\calP|k) \cong \G(\calP) \ltimes \langle \alpha, \beta \rangle$.
		\item The symmetry type graph of $\calP|k$ is obtained from the symmetry type graph of $\calP$ by adding
		semi-edges labeled $n$ to each node. In particular, $\calP|k$ is regular if and only if $\calP$ is regular.
		\end{enumerate}		
		\end{proposition}
		
		\begin{proof}
		The first part was already proved in the previous discussion. For the second part, 
		let us first show that every automorphism in $\G(\calP|k)$ may be written as $\tilde{\varphi} \gamma$, with
		$\varphi \in \G(\calP)$ and $\gamma \in \langle \alpha, \beta \rangle$. Fix a base flag $\Phi$ of $\calP$,
		and suppose that an automorphism $\psi$ of $\calP|k$ sends $\Phi_1$ to $\Psi_j$. Then there must be
		an automorphism $\varphi$ of $\calP$ that sends $\Phi$ to $\Psi$, and the induced automorphism $\tilde{\varphi}$
		sends $\Phi_1$ to $\Psi_1$. Then there is some $\gamma \in \langle \alpha, \beta \rangle$ that sends $\Psi_1$
		to $\Psi_j$, and so $\tilde{\varphi} \gamma$ sends $\Phi_1$ to $\Psi_j$. Since polytope automorphisms
		are determined by their action on any one flag, this shows that $\psi = \tilde{\varphi} \gamma$.
		
		Next, we note that $\alpha$ and $\beta$ both only change the subscript of a flag independently of the underlying
		flag of $\calP$. Similarly, $\tilde{\varphi}$ only changes the underlying flag, independent of the subscript.
		So if $\gamma \in \langle \alpha, \beta \rangle$, then
		$\tilde{\varphi}^{-1} \gamma \tilde{\varphi}$ also only changes the subscript of each flag independently
		of the underlying flag, and so $\tilde{\varphi}^{-1} \gamma \tilde{\varphi} \in \langle \alpha, \beta \rangle$.
		So $\langle \alpha, \beta \rangle$ is normal in $\G(\calP|k)$. Finally, since each $\tilde{\varphi}$ fixes
		the first layer setwise, whereas no nontrivial element of $\langle \alpha, \beta \rangle$ fixes the first layer,
		we find that $\langle \alpha, \beta \rangle \cap \G(\calP) = \langle 1 \rangle$, and so $\G(\calP|k) \cong
		\G(\calP) \ltimes \langle \alpha, \beta \rangle$.
		
		For the last part, note that the orbit of each $\Lambda_i$ under $\langle \alpha, \beta \rangle$ consists of all $k$ flags of the form
		$\Lambda_j$, and so these flags are all identified under the quotient by $\G(\calP|k)$. In particular, each flag
		is in the same orbit as its $n$-adjacent flag. Furthermore, any pair of flags $\Phi_i$ and $\Psi_j$ that lie
		in the same orbit must have underlying flags $\Phi$ and $\Psi$ that lie in the same orbit of $\G(\calP)$,
		and so the symmetry type graph of $\calP|k$ is just the symmetry type graph of $\calP$ with extra semi-edges
		labeled $n$ at each node.
		\end{proof}

		Let us now show some nice properties of $\calP|k$ related to covers.
		
		\begin{proposition} \label{p:respect-covers}
		If $\calP$ and $\calQ$ are facet-bipartite polytopes such that $\calQ$ covers $\calP$,
		then $\calQ|k$ covers $\calP|k$ for every even positive integer $k$.
		\end{proposition}
		
		\begin{proof}
		To say that $\calQ$ covers $\calP$ is to say that there is a color-preserving graph epimorphism
		$\varphi$ from $\calQ$ to $\calP$. Fix a flag $\Psi$ of $\calQ$ and let $\Phi = (\Psi) \varphi$.
		Without loss of generality, we may color both $\Phi$ and $\Psi$ red, so that $\Phi_1$ is matched to
		$\Phi_2$ and $\Psi_1$ is matched to $\Psi_2$. Then the obvious extension of $\varphi$
		that acts separately on each layer of $\calQ|k$ will also respect the edges of color $n$, and thus
		$\calQ|k$ covers $\calP|k$.
		\end{proof}
		
		\begin{proposition}
		If $\calP$ is a facet-bipartite polytope and $k_1$ and $k_2$ are positive even integers with $k_2$
		a multiple of $k_1$, then $\calP|k_2$ covers $\calP|k_1$. In particular, for every even positive integer
		$k$, the polytope $\calP|k$ covers the trivial extension $\calP|2$.
		\end{proposition}
		
		\begin{proof}
		The function taking each $\Phi_i$ to $\Phi_{i \textrm{(mod $k_1$)}}$ is a color-preserving graph epimorphism.
		\end{proof}
	
		Next, we note that it is possible to repeatedly apply this construction:

		\begin{proposition}
		\label{iterated-extension}
		If $\calP$ is a facet-bipartite $n$-polytope, then for every finite sequence $k_1, \ldots, k_m$ with each
		$k_i$ a positive even integer, there is a facet-bipartite polytope $\calQ = \calP|k_1|k_2|\cdots|k_m$ that is $(i, i+2)$-flat
		for each $i$ in $\{n \!-\! 2, \ldots, n+m-3\}$. Furthermore, $\calQ$ is regular if $\calP$ is regular.
		\end{proposition}

		\begin{proof}
		The first part follows immediately from the fact that the facet graph of $\calP|k$ is an even cycle (consisting of the $k$ layers $\calP_i$), and so $\calP|k$ is
		facet-bipartite. The second part follows from \cref{pk-auts}(c).
		\end{proof}

		\begin{example}
		For any sequence of positive even integers $k_1, \ldots, k_m$, we can take $\calP$ to be a $k_1$-gon
		and then extend it by $k_2, \ldots, k_m$. This yields a regular $(m \!+\! 1)$-polytope that is $(i,i+2)$-flat for each $i$ in $\{0, \ldots, m \!-\! 1\}$.
		In fact, this is a \emph{tight polytope of type $\{k_1, \ldots, k_m\}$}; see \cite{tight-polytopes}.
		\end{example}

	\subsection{Flat extensions of other polytopes}
	\label{other-exts}

		We have seen that if $\calP$ is facet-bipartite, then there is a straightforward matching on $k$ copies of
		$\calP$ that yields a polytope $\calP|k$. What can we do with other polytopes $\calP$? Let us fix an even $k \geq 4$
		and try to build an $(n \!-\! 2, n)$-flat $(n \!+\! 1)$-polytope $\calQ$ with $k$-facets isomorphic to $\calP$.
		As before, we will focus on the case where each flag $\Phi_i$ is matched to some $\Phi_j$.
		
		First, recall that \cref{facet-coloring} says that in order for $\calQ$ to exist, $\calP$ must be $(k \!-\! 1)$-facet-colorable. 
		Naturally, we wonder whether this necessary condition is also sufficient. Suppose $\mu$ is a proper coloring of
		the facet graph of $\calP$, with colors $1, 2, \ldots, k \!-\! 1$, (though some colors may not be used). As before, we can extend this to a (non-proper) coloring
		of $\calP$ itself by coloring each flag according to the color of its facet. Take $k$ copies of $\calP$ as before: 
		$\calP_1, \ldots, \calP_k$, with each $\Phi_i$ colored the same as $\Phi$. For each color $c$, we designate
		a perfect matching $\sigma_c$ of the layers, and if $\Phi$ is color $c$, then we match $\Phi_i$ to
		$\Phi_{\sigma_c(i)}$. Since $\mu$ is a proper coloring of the facet graph, this ensures that flags in $\calP_k$ that 
		are $(n \!-\! 1)$-adjacent are matched to distinct layers, as required (see \cref{different-layers}).

		To determine whether the matchings $\sigma_c$ satisfy the desired properties, it is helpful to represent them
		using a new graph called the \emph{layer graph}. This is a graph on $k$ nodes, corresponding to the $k$
		layers $\calP_1, \ldots, \calP_k$, where there is an edge of color $c$ between two nodes if $\sigma_c$ matches
		the corresponding layers. See \cref{layer-graph-figure} for an example with $k = 6$.

		\begin{figure}[htbp]
		\begin{center}
		\includegraphics[height=4cm]{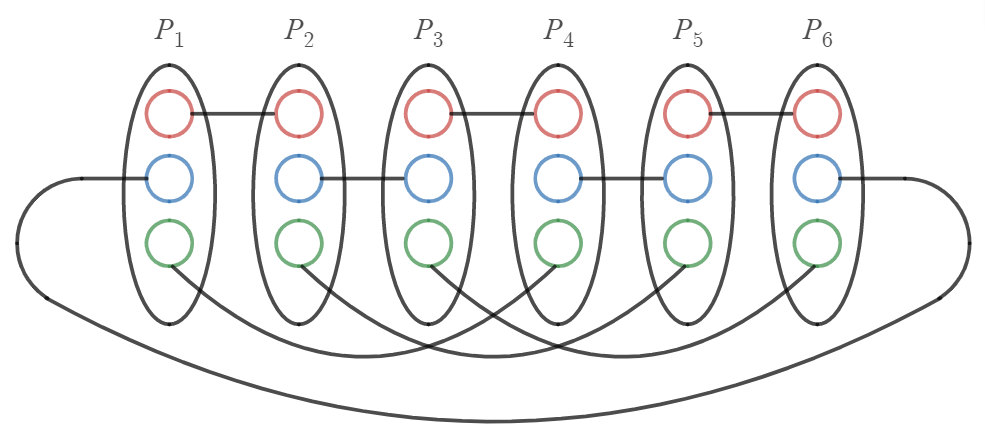} \\
		\includegraphics[height=4cm]{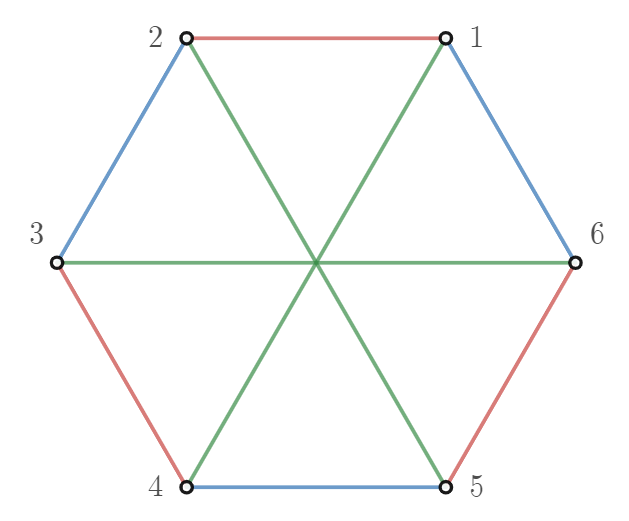}
		\caption{A matching of layers by color (above) and the corresponding layer graph (below).}
		\label{layer-graph-figure}
		\end{center}
		\end{figure}

		Our goal is to pick matchings so that we obtain an $(n \!-\! 2, n)$-flat $(n \!+\! 1)$-polytope.
		Recall that to be $(n \!-\! 2,n)$-flat means that, for every flag $\Phi_i$, the cycle that starts from $\Phi_i$ and follows
		edges labeled $n \!-\! 1$ and $n$ should intersect every layer. Note that such a cycle consists only of flags
		of the form $\Phi_j$ and $\Psi_j$, where $\Psi = \Phi^{n \!-\! 1}$. Therefore, the cycle is completely
		determined by the matchings corresponding to the colors of $\Phi$ and $\Psi$. Thus, if every pair of matchings
		of the layers yields a single cycle, then the result will be $(n \!-\! 2,n)$-flat. In terms of the layer graph, this means that
		it suffices for every pair of colors to yield a single cycle. Such a collection
		of matchings is called a \emph{perfect $1$-factorization} of the graph. Kotzig conjectured in
		1964 that every complete graph on an even number of vertices has a perfect $1$-factorization \cite{kotzig-conj}. This conjecture
		remains open; see \cite{perfect-1-factorizations} for a recent survey on this and related problems.

		In any case, let us suppose that the complete graph $K_k$ admits a perfect $1$-factorization, and match flags
		$\Phi_i$ accordingly. As discussed, this will give us something that is $(n \!-\! 2,n)$-flat. We still need to demonstrate
		that it is a polytope.
		
		\begin{theorem}
		\label{pk-polytope-2}
		Let $k$ be a positive even integer, $k \geq 4$, and let $\calP$ be $(k \!-\! 1)$-facet-colorable.
		Suppose that the complete graph $K_k$ has a perfect $1$-factorization. Then the preceding construction defines
		the flag graph of a polytope.
		\end{theorem}
		
		\begin{proof}
		Let $\calG$ be the graph defined above. First, let us show that it is connected. 
		The facet graph of $\calP$ must use at least two colors, and by construction, the matchings corresponding to those
		two colors must induce a cycle that intersects each layer. Since each layer is connected, this shows that $\calG$
		is itself connected.
		
		The remainder of the proof is analogous to the proof of \cref{pk-polytope}. The key element
		is that each $\Phi_i$ is matched to some $\Phi_j$ -- that is, each flag is matched to ``itself'' in another layer.
		\end{proof}

		\begin{example}
		If $k \!-\! 1$ is prime or $k/2$ is prime, then there is a perfect $1$-factorization of $K_k$; see \cite{kotzig-conj}
		and \cite{finite-topologies}, respectively. Thus, every finite polytope $\calP$ has infinitely many flat extensions --- 
		simply take $k \!-\! 1$ to be a prime that is greater than or equal to the number of facets of $\calP$.
		\end{example}
		
\section{Flat amalgamations}
\label{flat-amalg}

	There is another way of thinking about $\calP|k$ that readily admits one last generalization. It starts with seeing
	$\calP|k$ as a \emph{mix} of $\calP$ with the flag graph of a $k$-gon. A similar construction for regular
	polytopes was described in \cite[Sec. 4F]{arp}, using their automorphism groups instead of their flag graphs.
	For non-regular polytopes, the construction may provide different results depending on the choice of a base
	flag, and so we define the construction using \emph{rooted polytopes} $(\calP, \Phi)$ (see \cite{k-orbit}). 
	
	\begin{definition}
	Suppose that $\calP$ is an $n$-polytope with base flag $\Phi$ and that $\calQ$ is an $m$-polytope with base flag $\Psi$.
	Let $0 \leq r \leq n \!-\! 1$ with also $r \geq n-m$.
	Then the \emph{$r$-mix of $(\calP, \Phi)$ with $(\calQ, \Psi)$}, denoted $(\calP, \Phi) \mix_{r} (\calQ, \Psi)$,
	is the connected, properly edge-colored, $(m+r)$-regular graph $\calM$ defined as follows.
	\begin{enumerate}
	\item The base flag of $\calM$ is the pair $(\Phi, \Psi)$.
	\item For each $i \in \{0, \ldots, m+r-1\}$ and for each flag $(\Lambda, \Delta)$ of $\calM$, we define
	$(\Lambda, \Delta)^i$ to be $(\Lambda^i, \Delta^{i-r})$, with the understanding that if a superscript
	is ``out of bounds'' then we treat it as empty.  In other words:
	\[ (\Lambda, \Delta)^i = 
	\begin{cases}
	(\Lambda^i, \Delta) & \textrm{ if $0 \leq i < r$}, \\
	(\Lambda^i, \Delta^{i-r}) & \textrm{ if $r \leq i \leq n \!-\! 1$}, \\
	(\Lambda, \Delta^{i-r}) & \textrm{ if $n \leq i \leq m+r-1$}
	\end{cases}. \]
	\item The flags of $\calM$ are all pairs $(\Lambda, \Delta)$ (with $\Lambda$ a flag of $\calP$ and $\Delta$ a flag of $\calQ$)
	that are in the same connected component as $(\Phi, \Psi)$.
	\end{enumerate}
	\end{definition}

	\begin{definition}
	Suppose that $\calP$ is an $n$-polytope and that $\calQ$ is an $m$-polytope. Then the \emph{flat amalgamation of $(\calP, \Phi)$ with $(\calQ, \Psi)$} is
	$(\calP, \Phi) \mix_{n \!-\! 1} (\calQ, \Psi)$, denoted $(\calP, \Phi) | (\calQ, \Psi)$. If the base flags are understood in context, then we simply write
	$\calP | \calQ$. Note that, for each $i \in \{0, \ldots, m \!+\! n \!-\! 2\}$ and for each flag $(\Lambda, \Delta)$,
	\[ (\Lambda, \Delta)^i = 
	\begin{cases}
	(\Lambda^i, \Delta) & \textrm{ if $0 \leq i < n \!-\! 1$}, \\
	(\Lambda^{n \!-\! 1}, \Delta^{0}) & \textrm{ if $i=n \!-\! 1$}, \\
	(\Lambda, \Delta^{i-n \!+\! 1}) & \textrm{ if $n \leq i \leq m \!+\! n \!-\! 2$}
	\end{cases}. \]

	\end{definition}

	Recall that $\calP$ is facet-bipartite if and only if there are no cycles in $\calP$ with an odd number of edges labeled $n \!-\! 1$, and
	that $\calQ$ is vertex-bipartite if and only if there are no cycles in $\calQ$ with an odd number of edges labeled $0$.
	
	\begin{proposition}
	Let $\calP$ be an $n$-polytope with base flag $\Phi$ and let $\calQ$ be an $m$-polytope with base flag $\Psi$. Let $\calM = \calP | \calQ$.
	\begin{enumerate}
	\item Each connected component of $\calM[0, \ldots, n \!-\! 1]$ is isomorphic to $\calP$ if and only if $\calP$ is facet-bipartite.
	\item Each connected component of $\calM[n \!-\! 1, \ldots, m \!+\! n \!-\! 2]$ is isomorphic to $\calQ$ (with edge labels increased by $n \!-\! 1$) if and only if $\calQ$ is vertex-bipartite.
	\end{enumerate}
	\end{proposition}

	\begin{proof}
	Without loss of generality, consider the connected component of $\calM[0, \ldots, n \!-\! 1]$ that contains $(\Phi, \Psi)$. Recall that for $i < n \!-\! 1$ we have
	that $(\Lambda, \Delta)^i = (\Lambda^i, \Delta)$, and so each flag in this connected component has either the form $(\Lambda, \Psi)$ or $(\Lambda, \Psi^0)$.
	Now let $\pi: \calM \to \calP$ be the projection in the first coordinate, sending each $(\Lambda, \Delta)$ to $\Lambda$. Since $\calP$ is an $n$-polytope and
	we have edges of labels $0$ through $n \!-\! 1$, $\pi$ is surjective. Furthermore, $\pi$ will be injective (and thus bijective) if and only if there is no flag $\Lambda$
	such that both $(\Lambda, \Psi)$ and $(\Lambda, \Psi^0)$ are in the connected component. A path from $(\Lambda, \Psi)$ to $(\Lambda, \Psi^0)$ exists
	if and only if there is a cycle in $\calP$ that includes $\Lambda$ and has an odd number of edges labeled $n \!-\! 1$. Thus, $\pi$ is bijective if and only if
	no such cycle exists, which is to say if and only if $\calP$ is facet-bipartite.
	
	The proof of the second part is analogous.
	\end{proof}

	In the usual language of polytopes, we say that if $\calP$ is facet-bipartite and $\calQ$ is vertex-bipartite, then
	the $n$-faces of $\calP | \calQ$ are isomorphic to $\calP$ and the co-$(n \!-\! 2)$-faces are isomorphic to $\calQ$. 

	We now collect a few properties of $\calP | \calQ$. Let $\calF(\calM)$ denote the set of flags of the maniplex $\calM$.
	As in \cref{flat-exts1}, we can properly color the facet graph of $\calP$ with two colors, and then extend this coloring
	to the flag graph. Similarly, we can properly color the $1$-skeleton of $\calQ$ with two colors and extend this coloring
	to the flag graph.

	\begin{proposition}
	\label{flag-parity}
	Let $\calP$ be a facet-bipartite $n$-polytope with base flag $\Phi$ and let $\calQ$ be a vertex-bipartite $m$-polytope
	with base flag $\Psi$. Color the flags of $\calP$ red and blue according to a bipartition of its facet graph, and color the flags of
	$\calQ$ red and blue according to a bipartition of its $1$-skeleton, and let us assume that $\Phi$ and $\Psi$ are both red.
	\begin{enumerate}
	\item $\calF(\calP|\calQ) = \{ (\Lambda, \Delta) \in \calF(\calP) \times \calF(\calQ) : \textrm{ $\Lambda$ and $\Delta$ are the same color} \}.$
	\item $|\calF(\calP | \calQ)| = \frac{1}{2} |\calF(\calP)| \cdot |\calF(\calQ)|$.
	\item $\calP | \calQ$ is $(n \!-\! 2,n)$-flat.
	\end{enumerate}
	\end{proposition}
	
	\begin{proof}
	Suppose that $(\Lambda, \Delta)$ is a flag of $\calP | \calQ$. By the definition of $(\Lambda, \Delta)^j$, either both
	components change color (when $j = n \!-\! 1$) or neither component changes color. Since $\calP | \calQ$ consists of only
	those flags that are reachable from $(\Phi, \Psi)$, which are both red, it follows that all flags of $\calP | \calQ$ have the
	same color in both components. 
	
	Now, suppose that $\Lambda$ and $\Delta$ are arbitrary flags of $\calP$ and $\calQ$ (respectively) that are the same color.
	There is a path in $\calP$ from $\Phi$ to $\Lambda$, and this induces a path in $\calP | \calQ$ that uses only edges
	of colors in $\{0, \ldots, n \!-\! 1\}$. Such a path will either take $(\Phi, \Psi)$ to $(\Lambda, \Psi)$ or to $(\Lambda, \Psi^0)$.
	In the latter case, we may follow an additional edge labeled $n \!-\! 1$ to arrive at $(\Lambda^{n \!-\! 1}, \Psi)$.
	Now, there is a path in $\calQ$ from $\Psi$ to $\Delta$, and this induces a path in $\calP | \calQ$ that uses only edges
	of colors in $\{n \!-\! 1, \ldots, m \!+\! n \!-\! 2\}$. Such a path will take us from $(\Lambda, \Psi)$ or $(\Lambda^{n \!-\! 1}, \Psi)$ to
	$(\Lambda, \Delta)$ or $(\Lambda^{n \!-\! 1}, \Delta)$. By the previous paragraph, since $\Lambda^{n \!-\! 1}$ has a different
	color to $\Delta$, the flag $(\Lambda^{n \!-\! 1}, \Delta)$ cannot be in $\calP | \calQ$, and so we have found a path from
	$(\Phi, \Psi)$ to $(\Lambda, \Delta)$, proving that the latter is a flag of $\calP | \calQ$. The second part follows immediately from the first.
	
	For the third part, we need to show that, given flags flags $(\Phi, \Psi)$ and $(\Lambda, \Delta)$ of $\calP | \calQ$, 
	there is a path from $(\Phi, \Psi)$ to $(\Lambda, \Delta)$ that can be written as the concatenation of a path that
	never uses color $n$ with a path that never uses color $n \!-\! 2$. The path described in the previous paragraph
	already satisfies this condition.
	\end{proof}

	\begin{theorem}
	\label{flat-amalg-polytope}
	Let $\calP$ be a facet-bipartite $n$-polytope and let $\calQ$ be a vertex-bipartite $m$-polytope.
	Let $\calM = \calP | \calQ$. Then $\calM$ is an $(m \!+\! n \!-\! 1)$-polytope that is $(n \!-\! 2,n)$-flat.
	\end{theorem}
	
	\begin{proof}
	It is straightforward to check that if $i$ and $j$ are in $\{0, \ldots, m \!+\! n \!-\! 2\}$ with $|i-j| > 1$, then $\calM[i,j]$ consists of $4$-cycles; this shows that
	$\calM$ is a maniplex. Flatness was proved in \cref{flag-parity}.
	To show that $\calM$ is a polytope, it suffices to show that it satisfies the Path Intersection Property. 
	Consider two arbitrary flags of $\calM$, say $(\Phi, \Psi)$ and $(\Lambda, \Delta)$. Suppose that there is a path from $(\Phi, \Psi)$
	to $(\Lambda, \Delta)$ that uses only colors in $\{0, \ldots, j\}$ and another path that uses only colors in $\{i, \ldots, m \!+\! n \!-\! 2\}$. We want
	to show that there must be a path that uses only the colors $\{i, \ldots, j\}$. 
	
	Since colors greater than $n \!-\! 1$ do not affect the first component, the path that uses colors in $\{i, \ldots, m \!+\! n \!-\! 2\}$ induces a path in $\calP$
	from $\Phi$ to $\Lambda$ that uses colors in $\{i, \ldots, n \shm 1\}$. Since colors less than $n \!-\! 1$ do not affect the second component,
	following the same sequence of colors in $\calM$ gives us a path from $(\Phi, \Psi)$ to either $(\Lambda, \Psi)$ or $(\Lambda, \Psi^0)$.
	In the latter case, we can follow one more edge of color $n \!-\! 1$ to arrive at $(\Lambda^{n \!-\! 1}, \Psi)$. Now, the path from $(\Phi, \Psi)$
	to $(\Lambda, \Delta)$ that uses colors in $\{0, \ldots, j\}$ induces a path from $\Psi$ to $\Delta$ that uses colors in
	$\{n \!-\! 1, \ldots, j\}$, and following this sequence of colors in $\calM$ gives us a path from wherever we stopped
	(either $(\Lambda, \Psi)$ or $(\Lambda^{n \!-\! 1}, \Psi)$) to either $(\Lambda, \Delta)$ or $(\Lambda^{n \!-\! 1}, \Delta)$. Since we supposed
	that $(\Lambda, \Delta)$ was a flag of $\calM$, \cref{flag-parity} implies that $(\Lambda^{n \!-\! 1}, \Delta)$ is not a flag of $\calM$, and so
	we must have arrived at $(\Lambda, \Delta)$. Thus, we have a path from $(\Phi, \Psi)$ to $(\Lambda, \Delta)$ that only uses
	colors in $\{i, \ldots, n \!-\! 1\} \cup \{n \!-\! 1, \ldots, j\} = \{i, \ldots, j\}$, as desired.
	\end{proof}

	\begin{example}
	If $\calQ$ is a $k$-gon with $k$ even, then $\calP | \calQ \cong \calP | k$. Essentially, each flag of the $k$-gon
	corresponds to a choice of one of the $k$ layers and one of the colors red or blue.
	\end{example}

	\begin{proposition} \label{iterated-2}
	Let $\calP$ be a facet-bipartite $n$-polytope and let $\calQ$ be a vertex-bipartite $m$-polytope.
	If $\calQ$ is facet-bipartite, then $\calP | \calQ$ is facet-bipartite.
	\end{proposition}
	
	\begin{proof}
	If there is a cycle in $\calP | \calQ$ with an odd number of edges labeled $m \!+\! n \!-\! 2$, this induces a cycle
	in $\calQ$ with an odd number of edges labeled $m \!-\! 1$.
	\end{proof}
	
	\cref{iterated-2} implies that, if $\calQ_1, \ldots, \calQ_k$ are all vertex-bipartite and facet-bipartite,
	then we may construct a flat amalgamation $\calP | \calQ_1 | \cdots | \calQ_k$.
	
	Finally, let us determine the automorphism group of $\calP | \calQ$. Given an automorphism $\varphi$ of
	$\calP$ that sends $\Phi$ to $\Lambda$, let us say that $\varphi$ is \emph{$(n \!-\! 1)$-even} (respectively $(n \!-\! 1)$-odd)
	if the number of edges labeled $n \!-\! 1$ in any path from $\Phi$ to $\Lambda$ is even (respectively odd). 
	(As long as $\calP$ is facet-bipartite, this is well-defined.) We will similarly define automorphisms of $\calQ$ to be
	$0$-even or $0$-odd.

	\begin{theorem}
	Let $\calP$ be a facet-bipartite $n$-polytope with base flag $\Phi$ and let $\calQ$ be a vertex-bipartite $m$-polytope
	with base flag $\Psi$. Then
	\[ \G(\calP | \calQ) = \{ (\varphi, \psi) \in \G(\calP) \times \G(\calQ) : \textrm{ $\varphi$ is $(n \!-\! 1)$-even if and only if $\psi$ is $0$-even} \}. \]
	In particular, if all automorphisms of $\calP$ are $(n \!-\! 1)$-even and all automorphisms of $\calQ$ are $0$-even, then $\G(\calP | \calQ) = \G(\calP) \times \G(\calQ)$,
	and otherwise $\G(\calP|\calQ)$ is an index-$2$ subgroup of $\G(\calP) \times \G(\calQ)$.
	\end{theorem}
	
	\begin{proof}
	Clearly, each automorphism of $\G(\calP | \calQ)$ induces an automorphism $\varphi$ of $\calP$ and
	an automorphism $\psi$ of $\calQ$, and so $\G(\calP | \calQ) \leq \G(\calP) \times \G(\calQ)$.
	Conversely, given automorphisms $\varphi$ and $\psi$, we may
	try to build an automorphism $(\varphi, \psi)$ of $\calP | \calQ$ that acts component-wise. Clearly, this
	will only work if $(\Phi \varphi, \Psi \psi)$ is in $\calP | \calQ$, and this is true if and only if the parity of
	the number of edges labeled $n \!-\! 1$ from $\Phi$ to $\Phi \varphi$ is the same as the parity of the number
	of edges labeled $0$ from $\Psi$ to $\Psi \psi$. If that is the case, then note that for
	each flag $(\Lambda, \Delta)$,
	\[ (\Lambda, \Delta)^i (\varphi, \psi) = (\Lambda^i, \Delta^{i-n \!+\! 1}) (\varphi, \psi) =(\Lambda^i \varphi, \Delta^{i-n \!+\! 1} \psi) =  ((\Lambda \varphi)^i, (\Delta \psi)^{i-n \!+\! 1}) = (\Lambda \varphi, \Delta \psi)^i, \]
	proving that $(\varphi, \psi)$ is an automorphism. That proves the first part and the second follows immediately.
	\end{proof}

	\begin{example}
	Suppose $\calP$ is the cuboctahedron and $\calQ$ is its dual, the rhombic dodecahedron. Then $\calP$ is facet-bipartite: we can color all of
	the square faces with one color and the triangles with another. Every automorphism of $\calP$ is $2$-even. Similarly, $\calQ$ is vertex-bipartite,
	and its automorphisms are all $0$-even. Thus $\G(\calP|\calQ) = \G(\calP) \times \G(\calQ)$, a group of order $48^2$.
	\end{example}

\section{Conclusions}	
\label{open-q}

	We have shown that every finite polytope $\calP$ has a flat extension, where we glue together an even number of
	copies of $\calP$ in a flat way. The strategy used does not work if we want to use an odd number of copies of
	$\calP$. In particular, if we use an odd number of copies, then we cannot match each flag $\Phi_i$ to some
	$\Phi_j$ --- some flags $\Phi_i$ must get matched to $\Psi_j$ with $\Phi \neq \Psi$. When is this possible
	and how can we do this in a consistent way? 
	
	\begin{problem}
	Describe a construction that takes an $n$-polytope $\calP$ and produces an $(n \!-\! 2,n)$-flat $(n \!+\! 1)$-polytope
	with $3$ facets all isomorphic to $\calP$. What restrictions on $\calP$ are there?
	\end{problem}
	
	Another interesting problem would be to further investigate the properties of the flat extensions that were
	described in \cref{other-exts}.
	
	\begin{problem}
	Determine the automorphism group of the flat extensions described in \cref{other-exts}.
	\end{problem}
	
\bibliographystyle{amsplain}
\bibliography{gabe}

\end{document}